\documentclass{article}
\usepackage{graphicx} 

\usepackage[noadjust]{cite}
\usepackage{xcolor}
\usepackage{url}
\usepackage{latexsym,epsfig,amsfonts}
\usepackage{authblk}
\usepackage{tikz} 
\usepackage{verbatim}

\usepackage{amsthm}
\usepackage{amsmath}

\RequirePackage[all]{xy}

\newtheorem{thm}{Theorem}[section]

\newtheorem{prop}[thm]{Proposition}

\newtheorem{obs1}[thm]{Observation}
\theoremstyle{definition}

\theoremstyle{remark}

\numberwithin{equation}{section}

\newcommand{\BibTeX}{B\kern-0.1emi\kern-0.017emb\kern-0.15em\TeX}
\newcommand{\XYpic}{$\mathrm{X\kern-0.3em\raisebox{-0.18em}{Y}}$-$\mathrm{pic}\,$}

\newcommand{\cl}{C \kern -0.1em \ell}  

\newcommand{\ed}{\end{document}}
\begin{document}

%
%
%
%
%
%
%
%
%

\title
 {Variety of mutual-visibility problems in hypercubes}
\author[$\dagger$]{Danilo Kor\v ze}
%
\affil[$\dagger$]
{%
Faculty of Electrical Engineering and Computer Science\\
University of Maribor\\
Koro\v ska cesta 46\\
 SI-2000 Maribor, Slovenia }
%
%
\author[$\star$]{Aleksander Vesel}
\affil[$\star$]{%
Faculty of Natural Sciences and Mathematics\\
University of Maribor\\
Koro\v ska cesta 160\\
 SI-2000 Maribor, Slovenia}
%
%
\date{\today}
\maketitle

\begin{abstract}
Let $G$ be a graph and $M \subseteq V(G)$.  Vertices $x, y \in M$ are $M$-visible if there exists a shortest 
$x,y$-path of $G$ that does not pass through any vertex  of $M \setminus \{x, y \}$. 
We say that $M$ is a mutual-visibility set if each pair of vertices of $M$ is $M$-visible,  while
the size of any largest mutual-visibility set of $G$ is the mutual-visibility number of $G$.
If some additional combinations for pairs of vertices  $x, y$  are required to be $M$-visible,
we obtain the total  (every $x,y  \in  V(G)$ are $M$-visible), the outer 
(every $x \in M$ and  every $y \in V(G) \setminus M$ are $M$-visible),  
and the dual  (every $x,y \in V(G) \setminus M$ are $M$-visible)
mutual-visibility set  of $G$. 
The cardinalities of the largest of the above defined sets are known as the total, the outer, and the dual  mutual-visibility number of $G$, respectively.

We  present results on the variety of  mutual-visibility problems in  hypercubes.
\end{abstract}
\label{page:firstblob}

\section{Introduction and preliminaries}

In graph theory, a mutual-visibility set refers to a collection of vertices within a graph where for every pair of vertices there exists a shortest path between them that avoids any other vertex in the set. The mutual-visibility number of a graph represents the maximum size of such a mutual-visibility set within the graph.

Mutual-visibility sets have been examined across diverse fields, encompassing wireless sensor networks, mobile robot networks, and distributed computing. In wireless sensor networks, they aid in sensor deployment to ensure interference-free communication among sensors. In mobile robot networks, they facilitate collision avoidance strategies for robot control. In distributed computing, they contribute to the development of efficient algorithms for tasks such as consensus and broadcasting \cite{Aljohani, Bhagat, Cicerone4, DiLuna, Poudel}.
In 2022, Di Stefano \cite{DiS} established the foundation of the preceding problem within graph theory, aiming to maximize the size of the largest mutual-visibility set. 

Research highlighted in the cited papers explores visibility problems in graphs, indicating specific adjustments to visibility properties of significance. In this context, \cite{Cicerone2} presents a range of novel mutual-visibility problems, including the total mutual-visibility problem, the dual mutual-visibility problem, and the outer mutual-visibility problem.
The total mutual-visibility number on Cartesian products of complete
graphs, also known as Hamming graphs, is studied  in  \cite{Tian}. 
The (total) mutual-visibility sets in hypercubes, a subset of Hamming graphs, are investigated in \cite{Cicerone} as well as in 
\cite{Axenovich}. 

In this paper, we build upon previous research of mutual-visibility on hypercubes. Subsequently, we present definitions and results crucial for the remainder of the paper in the following section. Section 2 delves into mathematical optimization techniques capable of yielding novel insights into the mutual-visibility set problem in hypercubes. In Section 3, we investigate the mutual-visibility number and its variations within hypercubes. We establish new upper bounds for these numbers and provide exact values for hypercubes of smaller dimensions. Additionally, we extend the concept introduced in \cite{Axenovich}, demonstrating that the upper bound on the total mutual-visibility number in hypercubes can be determined by leveraging the size of the largest binary code. This is achieved by showing that the task of finding the largest total mutual-visibility set in an $h$-cube is equivalent to determining the maximum size of a binary code of length $h$ with a minimum Hamming distance of 4.

Let $G = (V(G),E(G))$ is a graph, $M \subseteq V(G)$ and $u,v \in V(G)$.
We say that a  $u,v$-path $P$ is {\em $M$-free}, 
 if $P$ does not contain a vertex of $M \setminus \{u, v \}$.
Vertices $u, v \in V(G)$ are {\em $M$-visible}  if $G$ admits an $M$-free shortest $u,v$-path.
 
Let $M \subseteq V(G)$ and $\overline M=V(G) \setminus M$. Then we say that $M$ is a

 \begin{itemize}
  \item \emph{mutual-visibility} set, if every $u, v \in M$ are $M$-visible,

 \item \emph{total mutual-visibility} set, if every $u, v \in V(G)$ are $M$-visible,

 \item \emph{outer mutual-visibility} set, if every $u, v \in M$ are $M$-visible, and every $u \in M$, $v \in \overline M$ are $M$-visible,

\item \emph{dual mutual-visibility} set, if every $u, v \in M$ are $M$-visible, and every $u, v \in \overline M$ are $M$-visible.
 \end{itemize}
 
The cardinality of a largest mutual-visibility set, a largest total mutual-visibility set, a largest outer mutual-visibility set, and a largest dual mutual-visibility set will be respectively denoted by $\mu(G)$, $\mu_t(G)$, $\mu_o(G)$, and $\mu_d(G)$. 
Also, these graph invariants will be respectively called the {\em mutual-visibility number}, the {\em total mutual-visibility number}, the {\em outer mutual-visibility number}, and the {\em dual mutual-visibility number} of $G$. 
 

Let $B = \{0, 1\}$. If $b$ is a word of length $h$ over $B$, that is, 
$b = (b_1, \ldots, b_h) \in B^h$,
then we will briefly write $b$ as $b_1 \ldots b_h$.
Elements of $B^h$ are also called binary strings of length $h$.
If $x, y \in B^h$, 
then the {\em Hamming distance} $H(x,y)$ between $x$ and $y$ is the number 
of positions in which $x$ and $y$ differ. 

The {\em  hypercube} of order $h$ or simply {\em  $h$-cube}, 
denoted by $Q_h$, is the graph $G=(V,E)$, where the vertex
set $V(G)$ is the set of all binary strings of length $h$, while
 two vertices $x,y \in V(G)$ are adjacent in $Q_h$ if and only if 
the Hamming distance between $x$ and $y$ is equal to one.

The {\em weight} of $u \in B^h$ is 
$w(u) = \sum_{i=1}^h u_i$, in
other words, $w(u)$ is the number of 1s in the word $u$. 
For the concatenation of bits the power notation will be used, for instance 
$0^h = 0\ldots 0 \in  B ^h$.

If $G$ is a connected graph, then the distance $d_G(u, v)$ (or simply $d(u,v)$)  between vertices $u$ and $v$ is the length of a shortest $u,v$-path 
(that is, a shortest path between $u$ and $v$) in $G$. 
The set of vertices lying on all shortest $u, v$-paths is called
the {\em interval} between $u$ and $v$ and denoted by $I_{G}(u, v)$ 
 We will also write $I(u, v)$ when $G$ will be clear from the context.

If $G$ is a graph and $X \subseteq V(G)$, then $G[X]$ 
denotes the subgraph of $G$ induced by $X$.

If $u$ is a vertex of a graph $G$, let $N_G(u)$  (or simply $N_G(u)$)
denote the set of neighbors of $u$. Moreover, let $N[u] = N(u) \cup \{  u \}$. 

For a positive integer $n$  we will use the notation $[n] = \{1, 2, \ldots, n\}$.

The  \emph{Cartesian product} of graphs
$G$ and $H$ is the graph $G \Box H$  with vertex set $V(G) \times V(H)$
and $(x_1,x_2)(y_1,y_2) \in E(G \Box H)$  whenever  $x_1y_1 \in E(G)$ and $x_2=y_2$,  or
 $x_2y_2 \in E(H)$ and $x_1=y_1$. 
 It is well-known, that the Cartesian product is commutative and associative, having the trivial graph as a unit. 
 
Since for $h\ge 2$ we have $Q_h = Q_{h-1} \Box  K_2$, it was observed  
in \cite{Cicerone} that the mutual-visibility number of $Q_h$ is bounded above
by  twice the mutual-visibility number of $Q_{h-1}$.
Clearly, this observation can be generalized as follows.

\begin{obs1} \label{observation}
If $h \ge 2$, then 

(i) $\mu(Q_h) \le 2 \mu(Q_{h-1})$,

(ii) $\mu_t(Q_h) \le 2 \mu_t(Q_{h-1})$,

(iii) $\mu_o(Q_h) \le 2 \mu_o(Q_{h-1})$,

(vi) $\mu_d(Q_h) \le 2 \mu_d(Q_{h-1})$.

\end{obs1}

We will need the following well known result,  see for example \cite{imkl-00}.

\begin{prop} \label{kocke}
Let  $u, v \in V(Q_h)$.

(i) If $u_i=v_i$ for some $i \in [h]$,  then $x_i=u_i=v_i$ for every $x \in I(u,v)$.

(ii) If $d(u,v)= d$, then the subgraph of $Q_h$ induced by $I(u,v)$ is isomorphic to the $d$-cube.

\end{prop}

The {\em diameter}  $diam(G)$ of a connected graph $G$ is the maximum distance between two vertices of $G$. 
Let $G$ be a graph, $d \le diam(G)$ and $M \subseteq V(G)$. 
Vertices $u, v \in V(G)$ are {\em $M_d$-visible}  if $d(u,v) > d$ or $G$ admits an $M$-free shortest $u,v$-path.
It is evident that  $M$ is not a mutual-visibility set if $M$ admits a pair of vertices that are not $M_d$-visible. 
Furthermore, this concept can be naturally extended to 
total,  outer and dual mutual-visibility sets of $G$.  
of $G$.  

\section{Mathematical optimization methods }
We applied two well-known  techniques for computing the mutual-visibility sets of
hypercubes: Integer linear programming (ILP) and a reduction to SAT.
Both methods have been previously successfully applied for
distance constrained coloring problems for various finite and infinite graphs. used for distance-constrained coloring problems in various finite and infinite graphs. For instance, \cite{radio} demonstrates the application of these concepts in the radio coloring problem.

\subsection{Integer linear programming (ILP) model}

Let  $G=(V, E)$ be a graph.
For a vertex $v \in V(G)$ we introduce the Boolean variable $x_v$ such that $x_v=1$ if and only if $v$ belongs to the mutual-visibility set $M$ of $G$.
The problem of finding the maximal mutual-visibility set can be formulated as an integer linear program as follows:\\

\begin{equation}
\textrm{maximize } \sum_{v \in V(G)} x_v
\end{equation}
subject to:
\begin{eqnarray}
 x_u + x_v - \sum_{P \in {\cal P}(u,v)}{z_{u,v,P}} \leq 1,  & \forall u,v \in V(G); \label{lpeq1} \\
 z_{u,v,P} + x_{z} \leq 1, & \forall u,v \in V(G); \label{lpeq2} \\
  & \forall P \in {\cal P}(u,v); \forall z \in V(P) \setminus \{u,v\}. \nonumber  
\end{eqnarray}

where ${\cal P}(u,v)$ denotes the set of all different shortest paths between vertices $u$ and $v$, while for $P \in {\cal P}(u,v)$ the set 
$V(P)\setminus\{u,v\}$ comprises all
intermediary vertices in the corresponding shortest path $P$ between $u$ and $v$.
Clearly, all shortest paths between all pairs of vertices in  $G$ should be computed in order to establish the model for $G$. The additional variable $z_{u,v,P}$ equals 1 if and only if the  shortest path $P$ between $u$ and $v$ enables mutual visibility of these two vertices. That is to say, for each pair of vertices $u$ and $v$, the number of corresponding additional variables equals the number of different shortest paths between $u$ and $v$.
\\
Integer linear programming was also employed  for  searching the 
dual mutual-visibility and outer mutual-visibility sets of hypercubes.

For the dual visibility sets, the following additional constraints have to be added to constraints (\ref{lpeq1}):
\begin{equation}
 - x_u - x_v - \sum_{P \in {\cal P}(u,v)}{z_{u,v,P}} \leq -1, \   \forall u,v \in V(G); \label{lpeq3} 
\end{equation}

In order to compute outer visibility sets,  constraints (\ref{lpeq1}) has to be replaced with constraints (\ref{lpeq4}):
\begin{equation}
 x_u + x_v - 2\sum_{P \in {\cal P}(u,v)}{z_{u,v,P}}  \leq 0, \   \forall u,v \in V(G); \label{lpeq4} 
\end{equation}

Constraints (\ref{lpeq2}) are applied for all three varieties of mutual-visibility set. 

To reduce the computation time, we can apply equations only for paths shorter than a predefined length. In this scenario, we can only confirm the non-existence of a visibility set of a given cardinality. However, it has been observed that solutions obtained with path length restrictions often also constitute mutual-visibility sets, provided that the distance length limitation is not too profound.

We applied  Gurobi Optimization \cite{gurobi} for solving ILP models.


\subsection{Reduction to SAT model}
Let  $G=(V, E)$ be a graph with $n$ vertices and $\ell$ a positive integer 
(the size of a potential mutual-visibility set).
For every $v \in V(G)$ we introduce an atom $x_v$.
Intuitively, this atom expresses whether vertex $v$ is inside the mutual-visibility set $M$ or not. 
More precisely, $x_v=0$ if and only if $v$ belongs to the mutual-visibility set $M$.

First collection of propositional formulas define an encoding for cardinality constraints (known as $\ge k(x_1,...,x_n)$ constraints) which ensure that at most $k=n-\ell$ atoms are assigned value 1.  
We do not give the details of the applied encoding here;
the interested reader is referred to \cite{sinz}, where 
 the implemented encoding  based on the so called sequential counter  was introduced. 
 
In order to define mutual-visibility constraints, 
consider the following propositional formulas:

\begin{eqnarray}
x_u \vee  x_v \vee \bigvee_{P \in {\cal P}(u,v)} \left[ \bigwedge_{ x_z \in V(P) \setminus \{u,v\}} x_{z} \right] \  (\forall u,v  \in V(G))  
\label{sateq}
\end{eqnarray}

where ${\cal P}(u,v)$ denotes the set of all different shortest paths between vertices $u$ and $v$, while for $P \in {\cal P}(u,v)$ the set 
$V(P)\setminus\{u,v\}$ comprises all
intermediary vertices in the corresponding shortest path $P$ between $u$ and $v$.

Clearly, all shortest paths between all pairs of vertices in  $G$ should be computed in order to establish the clauses for $G$.
Before using the above formulas  in a SAT solver, they have to be transformed to the conjunctive normal form.
These  propositional formulas transform a mutual-visibility problem into a propositional satisfiability test (SAT).  We can confirm that a SAT instance is satisfiable if and only if $G$ has a mutual-visibility set of 
size at least $\ell$.
Note that for a given $h$-cube only cardinality constrains 
$\ge k(x_1,...,x_n)$  depend on $\ell$.

In the case of outer and dual visibility set problem, the  formulas (\ref{sateq}) have to be slightly modified. For the outer visibility set, we 
obtain

\begin{eqnarray}
x_u \wedge  x_v \vee \bigvee_{p \in {\cal P}(u,v)} \left[ \bigwedge_{ x_z \in V(P) \setminus \{u,v\}} x_{z} \right] \  (\forall u,v  \in V(G))  \label{sateqo}
\end{eqnarray}

while for the dual visibility set problem the clauses are of the following form:

\begin{eqnarray}
(x_u \wedge \neg x_v) \vee (\neg x_u \wedge x_v) \vee \bigvee_{P \in {\cal P}(u,v)} \left[ \bigwedge_{ x_z \in V(P) \setminus \{u,v\}} x_{z} \right] \  (\forall u,v  \in V(G)).  \label{sateqd}
\end{eqnarray}

To keep the model smaller and the computation more efficient, we may include in formulas only the paths of length up to given constant $s$.
As we noted in the above subsection,  only the non-existence of 
a mutual-visibility set of a given cardinality can be confirmed 
if $s < diam(G)$.
However, the tests performed  for $h$-cubes showed that solutions obtained with path length restrictions very often also constitute mutual-visibility sets if $s$
is close to $h$.  
 
We used the Criptominisat5 SAT-solver  \cite{Cryptominisat} to find the solutions of the above presented clauses for hypercubes.


\section{Theoretical and computational results}

\subsection{Mutual-visibility}
The mutual-visibility number of hypercubes has been studied in \cite{Cicerone}. 
Obtained exact values and bounds are summarized in the following proposition.

\begin{prop} \label{mv1}
\begin{displaymath}
\mu(Q_h) =
        \left \{ \begin{array}{llll}
              2,  &  h = 1   \\
              3,  &  h = 2   \\
              5,  &  h = 3   \\
              9,  &  h = 4   \\
              16,  &  h = 5   \\
             \end{array}. \right.
\end{displaymath}
Moreover, if $h \ge 6$, then 
\(\max_{i\in [h-3]}\) $({h \choose i } + {h \choose i + 3}) \le   \mu(Q_h) \le 2^{h-1}$.\end{prop}

For a vertex $v \in V(G)$ and $i \ge 0$ let $L^v_i$ (also called an $i$-{\em layer with respect to $v$})
denote the set of vertices of $V(G)$ at distance $i$ from $v$.

It is shown in \cite{Cicerone}, that $L^v_i \cup L^v_{i+3}$ is a  mutual-visibility set of  $Q_h$ for every $v \in V(Q_H)$ and $ 1 \le i \le h-3$. 
We will show in the sequel that $L^v_i \cup L^v_{i+k}$ cannot form a  mutual-visibility set if $k \le 2$.

\begin{prop} \label{sosednji}
Let $h \ge 3$ and $v \in V(Q_h)$. If $1 \le i \le h-2$, then $L^v_i \cup L^v_{i+1}$ is not a  mutual-visibility set of  $Q_h$.
\end{prop}

\begin{proof}
Let $X := L^v_i \cup L^v_{i+1}$. 
We can set w.l.o.g. that $v = 0^h$. It follows that for every $u\in L^v_i$ and every $z \in L^v_{i+1}$ we have 
$w(u)=i$ and  $w(z)=i+1$. 
Since $|N(z) \cap  L^v_i|= i+1 < h$,  there exist a vertex  $u\in L^v_i $ such that $uz \not\in E(Q_h)$.  
Suppose that there exists an $X$-free  $z,u$-path $P$.  Let $x$ denote the vertex adjacent to $z$ in $P$.  
Obviously,   $x \in L^v_{i+2}$ or $x \in L^v_{i}$. Since $L^v_i \subset X$, 
 we have $x \in L^v_{i+2}$.  Moreover,  
$N(u) \subset  L^v_{i+1} \cup L^v_{i-1}$.
But since from Proposition  \ref{kocke} it follows that $P$ cannot posses a vertex from $L^v_{i-1}$, it follows that 
$P$ admits a vertex from $L^v_{i+1}$  and we obtained a contradiction.
\end{proof}

\begin{prop} \label{sosednji2}
Let $h \ge 3$ and $v \in V(Q_h)$. If $1 \le i \le h-3$,
 then  $L^v_i \cup L^v_{i+2}$ is not a  mutual-visibility set of  $Q_h$.
\end{prop}

\begin{proof}
Let $X := L^v_i \cup L^v_{i+2}$.
We can set w.l.o.g. that $v = 0^h$. It follows that for every $u\in L^v_i$ and every $z \in L^v_{i+2}$ we have 
$w(u)=i$ and  $w(z)=i+2$. 
Suppose that $u=0^{h-i}1^{i}$ and $z=1^{i+2}0^{h-i-2}$.
Since $h-i \ge 3$ and   $i+2 \ge 3$,  it holds that $d(u,z) > 2$.  
Suppose that there exists an $X$-free  $z,u$-path $P$.  Let $x$ denote the vertex adjacent to $z$ in $P$.
Obviously,   $x \in L^v_{i+3}$ or $x \in L^v_{i+1}$. 
If $x \in L^v_{i+1}$,  let $y \not = z$ denote the vertex adjacent to $x$ in $P$.  Note that $N(x) \subset  L^v_{i} \cup L^v_{i+2}$.  It follows that $y \in X$. Since by $d(u,z) > 2$ we have $y \not = u$, we obtain a contradiction.

If $x \in L^v_{i+3}$,  let $w$ denote the vertex adjacent to $u$ in $P$.
Since $N(u) \subset  L^v_{i+1} \cup L^v_{i-1}$, we have $w \in L^v_{i+1}$.  
But $N(w) \subset  L^v_{i} \cup L^v_{i+2} = X$ and 
we obtained a contradiction.
\end{proof}

Let $u\in V(Q_h)$ and $X \subseteq N(v)$ such that $|X|=d \ge 1$. 
Note that the vertices of $X \cup \{u\}$ belong to a unique $d$-cube, say  $Q$,   a subgraph of $Q_h$, 
which is induced by $I(u,v)$,  where $v$ is the vertex of $Q_h$ at distance $d-1$ from all vertices of $X$.
We will say that the sub-cube $Q$ is {\em raised} by $ X \cup \{u\}$. 

Let $M$ be a mutual-visibility set of  $Q_h$.
If $u \in M$,  then $u$ and the vertices of $M$ adjacent to $u$ restrict the  number of vertices of $M$ in the 
 corresponding sub-cube, as shown in the next proposition.

\begin{prop} \label{star}
Let $u$  be a vertex of a mutual-visibility set $M$ of  $Q_h$.   If $X = N[u] \cap M$
and $Q$ the sub-cube of $Q_h$ raised by $X$,  then    $V(Q) \cap M = X$. 
\end{prop}

\begin{proof}
As noted above,  the sub-cube of $Q_h$ raised by $ X$  is the subgraph of $Q_h$ induced by  
 $I(u,v)$,  where $v$ is the vertex of $Q_h$ at distance $d-1$ from all vertices of $X \setminus \{u \}$.
Since from Proposition \ref{kocke} then it follows that for every $z \in I(u,v)\setminus M$ a shortest $u,z$-path contains 
a vertex from $X\setminus \{u \}$, the proof is complete.
\end{proof}

The above proposition states that if  $M$ contains a subset $Y \subseteq M$ such that $Q_h[Y]$ is isomorphic to $K_{1,d}$ then
 the $d$-cube raised by $Y$ cannot have other vertices in $M$. 

The following corollary of the above proposition 
can be utilized for a computer search of large mutual-visibility sets of  $Q_h$.
 
 \begin{prop} \label{omejitev}
 Let $h \ge 6$.
If $Y$ is a subset of a mutual-visibility set $M$ of  $Q_h$  such that $Q_h[Y]$ is isomorphic to $K_{1,4}$,  then 
$|M| \le 2^{h-1} - 2$.
\end{prop}
 
\begin{proof}
Let $u$ denote the vertex of degree 4 in $Q_h[Y]$.
We can assume  without loss of generality  that $u=0^h$, while the other four vertices of $Y$ are $10^{h-1}$, $010^{h-2}$, 
$0010^{h-3}$ and $00010^{h-4}$.  

 The vertices of $Y$ raise the 4-cube, denoted as $Q^0$, where $Q^0$ is induced by the set of vertices of the form $b_1 b_2 b_3 b_40^{h-4}$,  $b_i \in \{0,1\}$. 
 Consider now the sequence of cubes  $Q^1 \ldots Q^{h-4}$ induced by the subsets of $V(Q_h)$ as follows:
 
 $V(Q^1) =  \{ b_1 b_2 \ldots b_{h-1} 1$,  $b_i \in \{0,1\} \}$,
 
 $V(Q^2) =  \{ b_1 b_2 \ldots b_{h-2} 10$,  $b_i \in \{0,1\} \}$,
 
: 
 
 $V(Q^{h-4}) =  \{ b_1 b_2 b_3 b_410^{h-5}$,  $b_i \in \{0,1\} \}$.

Note that for $i \in [h-4]$ the set $V(Q^i)$ possesses exactly $2^{h-i}$ vertices, while $V(Q^0)$ possesses 
exactly $2^{4}$ vertices.  Moreover, the vertices of $V(Q^0), \ldots, V(Q^{h-4})$ partition the vertices of $Q_h$.

 By Proposition  \ref{star}, the cube $Q^0$ contains exactly 5 vertices in $M$ (since $M \cap V(Q^0) = Y)$. 
 By Proposition \ref{mv1}, for $i \in [h-3]$ we have that  $|M \cap V(Q^i)| \le 2^{h-i-1}$, while 
 $|M \cap V(Q^{4})| \le 9 = 2^3+1$. 
 It follows that $M \le  2^{h-2} +  2^{h-3} + \ldots + 2^3 +6 = 2^{h-1} - 2$.

\end{proof}

\begin{thm} \label{mv2}
\begin{displaymath}
\mu(Q_h) =
        \left \{ \begin{array}{llll}
              2,  &  h = 1   \\
              3,  &  h = 2   \\
              5,  &  h = 3   \\
              9,  &  h = 4   \\
              16,  &  h = 5   \\
              32,  &  h = 6   \\
              59,  &  h = 7   \\
             \end{array}. \right.
\end{displaymath}
Moreover, if $h \ge 8$, then ${h \choose \lfloor \frac{h}{ 2} \rfloor - 1} + {h \choose \lfloor \frac{h}{ 2} \rfloor + 2} \le  \mu(Q_h) \le  59  \cdot 2^{h-7} $.
\end{thm}

\begin{proof}
The results for $h \le 5$ and the general lower bound are from Proposition \ref{mv1} (note that $i=\lfloor \frac{h}{ 2} \rfloor - 1$ maximizes
${h \choose i } + {h \choose i + 3}$ for  $i\ge 8$). 

For $Q_6$ and  $Q_7$ we found a mutual-visibility set with 32 and 
59 vertices, respectively, by using an ILP model. 
To improve the efficiency of computing a mutual-visibility set $M$, 
we imposed a constraint on the number of vertices in $M \cap N[u]$ to be 4 for every $u \in M$, as implied by Proposition \ref{omejitev}. 
These computations confirmed that the mutual-visibility number of $Q_6$ is 32.
However, for $Q_7$, the ILP model did not refute the existence of a larger mutual-visibility set, as the computations did not finish after more than one month.

To verify that a larger mutual-visibility set of $Q_7$ does not exist, we employed a reduction to SAT. 
Since the search space in a straightforward approach was too large, we imposed additional restrictions.

First, we searched for the largest mutual-visibility set of $Q_7$ using an ILP model while prohibiting triples 
$u,v,z \in M$ such that $Q_7[\{u,v,z\}]$ is isomorphic to $K_{1,2}$. We determined that the size of a mutual-visibility set of $Q_7$ with this restriction cannot exceed 49. 

The above result indicates that every mutual-visibility set $M$ of $Q_7$ of size greater than 49 must include at least one triple 
$u,v,z \in M$ such that $Q_7[\{u,v,z\}]$ is isomorphic to $K_{1,2}$. 
This fact allows us to fix the vertices 0000000, 0000001, 0000010 to be included  in $M$ in our SAT model. Additionally, we limited the length of the paths involved in the SAT formulas to 5.

The resulting model comprised 254268 variables and 874806 clauses. After computation lasting 962,000 seconds on a powerful desktop computer with 16 threads, we confirmed that the presented reduction is unsatisfiable. Consequently, we can conclude that a mutual-visibility set of 
 $Q_7$ with 60 vertices does not exist. 

Since the upper bound on $\mu(Q_h)$ for $h\ge 8$ follows from the fact that
$\mu_o(Q_7)=59$ and Observation 
\ref{observation}, the proof is complete.
\end{proof}

\subsection{Total mutual-visibility}

Let $V_1 \cup V_2$ be the partition of $V(Q_h)$ such that $V_1$ and $V_2$ comprise the set of vertices 
of $Q_h$ with even and odd weights, respectively. Then the {\em  halved cube } $Q_h^e$ (resp.  $Q_h^o$) is the 
graph with  $V(Q_h^e)=V_1$ (resp. $V(Q_h^o)=V_2$), where $u$ and $v$ are adjacent in  $Q_h^e$ (resp.  $Q_h^o$) 
if $d(u,v)_{Q_h}=2$ (see \cite{polkocke} for more details). 
Clearly, graphs $Q_h^e$ and $Q_h^o$ are isomorphic. 
Therefore, we can denote either of them as $Q_h'$ when it is more convenient.

An independent vertex set of a graph $G$ is a subset of the vertices of $G$ 
such that no two vertices in the subset are adjacent in $G$. The size of a largest independent set of $G$ is called the 
{\em independent number} of $G$ and denoted $\alpha(G)$. 

An {\em independent} vertex set of a graph $G$ is a subset of its vertices where no two vertices are adjacent. The size of the largest independent set of $G$ is referred to as the {\em independent number} of $G$, denoted as 
$\alpha(G)$.

The following result is (in more general form) given in \cite{Bujtas}.
\begin{thm} \label{sandi}
 $M$ is a total mutual-visibility set of $Q_h$ if and only if for every $u,v \in M$,  it holds that $d(u,v) \not = 2$.
\end{thm}

We will demonstrate that determining the largest total mutual-visibility set in an $h$-cube is tantamount to identifying the largest independent set in the corresponding halved cube.

\begin{prop} \label{pol}
If $h \ge 1$,  then  $\mu_t(Q_h) = 2\alpha(Q_h')$.
\end{prop}

\begin{proof}
Let $I_e$ (resp. $I_o$) be an independent set in  $Q_h^e$ (resp. $Q_h^o$). 
Since $Q_h'$ is connected, for every $u,v \in  I_e$ 
(resp. $u,v \in I_o$) we have $d_{Q_h}(u,v) \ge 4$. Furthermore for every $u \in  I_e$ and $v \in  I_o$ it holds 
that $d_{Q_h}(u,v)$ is odd. Hence, by Theorem \ref{sandi} it follows that $I_e \cup I_o$ is a  total mutual-visibility set of $Q_h$.
Additionally, since $I_e \cap I_o = \emptyset$, we have $|I_e \cup I_o| = |I_e|+|I_o| \le 2\alpha(Q_h') \le  \mu_t(Q_h)$.

Now, consider $M$ as a total mutual-visibility set of 
 $Q_h$ and let $M_e$ (resp. $M_o$) denote all vertices of $M$ with (resp. odd) weight. 
By Theorem  \ref{sandi}, for every  $u,v \in M_e$ (resp. $u,v \in M_o$) 
we have $d_{Q_h}(u,v) \not = 2$. Moreover, since all weights of vertices in 
$M_e$ (resp. $M_o$) are of the same parity, we have 
$d_{Q_h}(u,v) \ge 4$ for every  $u,v \in M_e$ (resp. $u,v \in M_o$).
It is evident that $M_e$  (resp. $M_o$) forms an independent set in $Q_h^e$ (resp. $Q_h^o$). 
Therefore, $ |M_e|+|M_o| \le \mu_t(Q_h) \le  2\alpha(Q_h')$, completing the proof. 
\end{proof}

If $C \subseteq B^n$ such that for every $x,y \in C$ we have $H(x,y) \ge d$, 
then we say that $C$ is a {\em binary code of length $n$ and minimum Hamming distance $d$}. 

Let $A(n,d)$ 
denote the maximum size of a binary code of length $n$ and minimum Hamming distance $d$. 
It is well known, e.g. \cite{Sloane}, that $A(n,4) = A(n-1,3)$ for every $n \ge 6$. 

\begin{prop} \label{kode}
If $h \ge 1$,  then  $\alpha(Q_h') = A(h,4)$.
\end{prop}

\begin{proof}
Let $C$ denote a binary code of length $h-1$ and minimum Hamming distance $3$.  
Let us define the sets $C_e$ and $C_o$ as follows:
$$C_e = \{ u0 \, | \, u \in C \, { \rm and} \, w(u)  \,  {\rm even} \}  \cup  \{ u1 \, | \, u \in C \,{ \rm and} \, w(u)  \, { \rm odd}  \}.$$
$$C_o = \{ u0 \, | \, u \in C \, { \rm and} \, w(u)  \,  {\rm odd} \}  \cup  \{ u1 \, | \, u \in C \,{ \rm and} \, w(u)  \, { \rm even}  \}.$$

Clearly, $C_e \subseteq V(Q_h^e)$ and $C_o \subseteq V(Q_h^o)$. 
Moreover, for every $u' \in C_e$, there exists 
exactly one $u \in C$ such that either $u' = u0$ or $u'=u1$. 
Since for every $u,v \in C$  we have $H(u,v) \ge 3$, for the corresponding vertices $u',v' \in C_e$ it holds 
$d_{Q_h}(u',v') \ge 3$. Moreover,  since $d(u',v')$ is even, we have 
$d_{Q_h}(u',v') \ge 4$. 
It follows that $C_e$ is an independent set in $Q_h^e$. Since the situation for  $C_o$ and 
 $Q_h^o$ is analogous, we have  $\alpha(Q_h') \ge A(n,4)$.
 
 Let $I$ be an independent set of $Q_h'$. Since for every $u, v \in I$ we
 have $d_{Q_h}(u,v) \ge 4$, it follows 
  that $I$ is  a binary code of length $h$ and minimum Hamming distance $4$.  
 Thus, $\alpha(Q_h') \le A(n,4)$. This assertion completes the proof.
\end{proof}

From Propositions \ref{pol} and \ref{kode} it immediately follows that the problem of finding a largest total  mutual-visibility set in a $h$-cube is equivalent to finding the maximum size of a binary code of length $h$ and minimum Hamming distance 4. 
 
\begin{thm}  \label{total}
If $h \ge 1$,  then  $\mu_t(Q_h) = 2\cdot A(h,4)$.
\end{thm}

Note that the exact values of $A(h,4)$ are known for every $h \le 16$, while 
for bigger dimension only upper and lower bounds have been obtained
(for more information on the subject please refer to \cite{Agrell, Agrell2}). 
Total mutual-visibility  numbers of $Q_h$, $ 2 \le h  \le 16$
are presented in Table 
\ref{t2}. 

\begin{table}[h]	
{\small
\begin{center}
\caption{ Total mutual-visibility numbers of $Q_h$. } \label{t2}
\bgroup
\def\arraystretch{1.4}
\begin{tabular}{|c||c|c|c|c|c|c|c|c|c|c|c|c|c|c| }
  \hline
  $h$  & 3 & 4 & 5 & 6 & 7 & 8 & 9 & 10 & 11 & 12 & 13 & 14 & 15 & 16 
  \\
  \hline
  $\mu_t(Q_h)$ 
   & 2 & 4 & 4 & 8 & 16 & 32 & 40 & 80 & 144 & 288 & 512 & 1024 & 2048 & 4096 
  \\
  \hline
 \end{tabular}
\egroup
\end{center}
}
\end{table}

\subsection{Outer mutual-visibility}

\begin{prop} \label{outer}
Let $h \ge 2$, $v \in V(Q_h)$ and $i \in [h]$.  Then $L^v_i$ is an outer  mutual-visibility set of  $Q_h$.
\end{prop}

\begin{proof}
It is shown in the proof of \cite[Theorem 1]{DiS} that $L^v_i $ is a   mutual-visibility set of  $Q_h$. 
Thus, we have to show that for every $z  \not \in  L^v_i$ and every $u \in  L^v_i$  there exists a $L^v_i$-free 
shortest $u,z$-path.  Let $z \in  L^v_j$,  $j\not=i$.  Obviously, $d(u,z) \ge  |i-j|$.

Suppose first that $j<i$. If $d(u,z)= i-j$, then it is not difficult to find a shortest $u,z$-path $P$ such that every vertex 
in $P$ apart of $u$ and $z$ belongs to $ L^v_k$, where $j < k < i$.  It follows that  $P$ is $L^v_i$-free.
If  $d(u,z) > i-j$, then let $S$ (resp.  $S'$) denote the set of indices from $[h]$, where for every $\ell \in S$ we have  
$u_\ell \not = z_\ell$ and $z_\ell=0$ (resp.  $z_\ell=1$).   Let $S=\{i_1,  \ldots, i_s\}$ and $S'=\{i_{s+1},  \ldots,  i_{s+t}\}$. 
Clearly,  $d(u,z)=s+t$.  
Let $x^0=u$, while  for $k \ge 1$ let  $x^k$ be obtained from  $u$ by replacing  $u_{i_{1}},  u_{i_{2}}, \ldots, u_{i_{k}}$ 
by $z_{i_{1}},  z_{i_{2}}, \ldots, z_{i_{k}}$.  We can see that $x^{s+t}=z$. 
We now construct a  shortest $u,z$-path $P$, where 
$P=x^0,x^1,\ldots, x^s,x^{s+1}, \ldots,x^{s+t}$.  Obviously, for every vertex  $y \in V(P) \setminus   \{u \}$ 
we have $y \in L^v_k$, where $k < i$. If follows that $P$ is $L^v_i$-free.

The proof for $j>i$ is analogous.
\end{proof}

\begin{thm} \label{omv_thm}
\begin{displaymath}
\mu_o(Q_h) =
        \left \{ \begin{array}{llll}
              2,  &  h = 1   \\
              2,  &  h = 2   \\
              4,  &  h = 3   \\
              6,  &  h = 4   \\
              12,  &  h = 5   \\
              24,  &  h = 6   \\
              40,  &  h = 7   \\
             \end{array}. \right.
\end{displaymath}
Moreover, if $h \ge 8$, then $ { h \choose \lfloor \frac{h}{ 2} \rfloor }  \le \mu(Q_h) \le  40 \cdot  2^{h-7} $.
\end{thm}

\begin{proof}
For $h \le 6$,  we found outer mutual-visibility numbers of $h$-cubes   by using an ILP model and a reduction to SAT as presented in Section 2. 

Moreover, we found an outer mutual-visibility set of $Q_7$ with 40 vertices. 
In order to confirm that an outer mutual-visibility set of cardinality 41 
does not exist, some additional restrictions where needed. 

First, we searched (by using an ILP model) a largest  outer mutual-visibility set $M$ of $Q_7$ such that pairs $u,v \in M$ with $uv \in E(Q_7)$ are forbidden and established that a largest  outer 
mutual-visibility set of $Q_7$ with this restriction is of size 36. 

The above result assures that every outer mutual-visibility set $M$ of $Q_7$ of size bigger than 36 must contain at least one pair 
$u,v \in M$ with $uv \in Q_7$. 
This fact allows us to preset the vertices 0000000 and 0000001 to be included  in $M$ in the applied SAT model. Moreover, we restricted the length of the paths involved in the SAT formulas to 4. 
The obtained model comprises 86889 variables and 230460 clauses.
We confirmed (the computation lasted 5562 seconds on a powerful desktop computer with 6 threads) that the  above reduction is unsatisfiable.
Hence, we may establish that an outer mutual-visibility set of $Q_7$ with 41 vertices does not exist. 

In order to prove the lower bound, remind that $L^v_i$ is an outer
mutual-visibility set for every $i \in [h]$ by Proposition \ref{outer}. 
If  $v=0^h$, then  $L^v_i$ 
contains vertices of $Q_h$ with exactly $i$ ones.  
It follows that $|L^v_i| = { h \choose  i  }$, which is maximized
for $i = \lfloor \frac{h}{ 2} \rfloor$.

Since $\mu_o(Q_7)=40$,  the upper bound on $\mu_o(Q_h)$ for $h\ge 8$ follows from Observation  \ref{observation}. This assertion completes  the proof. 
\end{proof}

\begin{figure}[bt] 
\centering
\scalebox{0.8}
{ \tikzstyle{nod}= [circle, draw,inner sep=0pt, minimum size=0.3cm]
\begin{minipage}{\textwidth}
\begin{center}
\begin{tikzpicture}
\pgfsetxvec{\pgfpoint{1.cm}{0.0cm}}
\pgfsetyvec{\pgfpoint{0.0cm}{1.0cm}}
\node[nod] at (0,0) [label=below:$0001$] (0001) {};
\node[nod] at (5,0) [label=below:$1001$] (1001) {};
\node[nod] at (0,5) [label=above:$0011$] (0011) {};
\node[nod] at (5,5) [label=above:$1011$] (1011) {};
\node[nod] at (2,2) [label=above:$0101$] (0101) {};
\node[nod] at (7,2) [label=below:$1101$] (1101) {};
\node[nod] at (2,7) [label=above:$0111$] (0111) {};
\node[nod] at (7,7) [label=above:$1111$] (1111) {};
\node[nod] at (2.4,1.4) [label=below:$0000$] (0000) {};
\node[nod] at (4.4,1.4) [label=below:$1000$] (1000) {};
\node[nod] at (2.4,3.4) [label=above:$0010$] (0010) {};
\node[nod] at (4.4,3.4) [label=below:$1010$] (1010) {};
\node[nod] at (3.4,2.4) [label=below:$0100$] (0100) {};
\node[nod] at (5.4,2.4) [label=below:$1100$] (1100) {};
\node[nod] at (3.4,4.4) [label=above:$0110$] (0110) {};
\node[nod] at (5.4,4.4) [label=below:$1110$] (1110) {};
\path (0011)
edge (1011)
edge (0111)
edge (0001)
(1001)
edge (0001)
edge (1101)
edge (1011)
(1111)
edge (1101)
edge (1011)
edge (0111)
(0010)
edge (1010)
edge (0110)
edge (0000)
(1000)
edge (0000)
edge (1100)
edge (1010)
(1110)
edge (1100)
edge (1010)
edge (0110);
\path[dashed] (0101)
edge (1101)
edge (0001)
edge (0111)
(0100)
edge (1100)
edge (0000)
edge (0110);
\path[dotted] (0000) edge (0001)
(0010) edge (0011)
(0100) edge (0101)
(0110) edge (0111)
(1000) edge (1001)
(1010) edge (1011)
(1100) edge (1101)
(1110) edge (1111);

\filldraw (2.4,1.4) circle (3pt);
\filldraw (0,0) circle (3pt);
\filldraw (3.4,4.4) circle (3pt);
\filldraw (2,7) circle (3pt);
\filldraw (4.4,3.4) circle (3pt);
\filldraw (5,5) circle (3pt);

\end{tikzpicture}
\end{center}
\end{minipage}
}
\caption{An outer mutual-visibility set of $Q_4$} \label{outer_fig}
\end{figure}
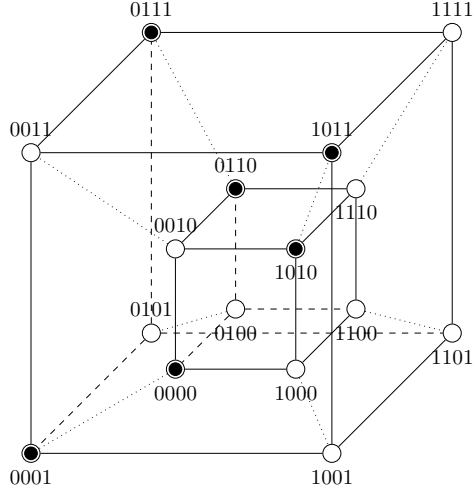

An example of a largest outer mutual-visibility set of $Q_4$ is presented in 
Fig. \ref{outer_fig}.

\subsection{Dual mutual-visibility}
Given a graph $G$, it is easy to see 
that every dual mutual-visibility set of $G$ is also a mutual-visibility set of $G$.

In this subsection, we first establish a more restricted relation between a mutual-visibility set and a dual  mutual-visibility set 
in hypercubes. More precisely, we present
 necessary and sufficient conditions  that a mutual-visibility set of  hypercube must satisfy in order to be a dual mutual-visibility set.

\begin{prop} \label{cube}
Let $M$ be a mutual-visibility set of $Q_h$.  Then $M$ is a dual mutual-visibility set of $Q_h$ if and only if 
for  every $u,v \in V(Q_h)$ with $d(u,v)=2$ it holds that  $|I(u,v) \cap M| \not=2$ or  $I(u,v) \cap M = \{w,  z\}$ and  $wz \in E(Q_h)$.
\end{prop}

\begin{proof}
If $M$ is a dual mutual-visibility set of $Q_h$,  then $M$ cannot admit vertices $u,v$ at distance 2 such that   
 $I(u,v) \cap M = \{u,  v\}$, as in this case,  for  $I(u,v) \setminus  M = \{x,  y\}$ we cannot find 
 an $M$-free shortest $x,y$-path.

Let for every 
$u,v \in V(Q_h)$ with $d(u,v)=2$ it holds that  $|I(u,v) \cap M| \not=2$ or  $I(u,v) \cap M = \{w,  z\}$ and  $wz \in E(Q_h)$.
Since $M$ is a mutual-visibility set of $Q_h$, for every $u,v \in V(Q_h)$ with $d(u,v)=2$ it holds that $|I(u,v)\cap M| \le 3$.
Now, we'll demonstrate that for every $x,y \in Q_h \setminus M$ there exists an $M$-free shortest $x,y$-path. 
If $d(x,y) \le 2$,  the claim is evident. Suppose, for contradiction,  that there exist $x,y \in \overline M$ with
 $d(x,y) > 2$ that are not $M$-visible.  Moreover,  let $x$ and $y$ be  vertices of $\overline M$
 that are not $M$-visible with the minimal distance $d(x,y)=i$. 
   
 By Proposition \ref{kocke} we have that $Q_h[I(x,y)]$ is isomorphic to $Q_i$. 
Let $L^x_j$ denote the set of vertices of $I(x,y)$ that are at distance $j$ from $x$. 
 
 We first show that vertices of $N_{Q_h[I(u,v)]}(y) = L^x_{i-1}$ must belong to $M$. Assuming otherwise, let $z\in L^x_{i-1}\setminus M$.
 By minimality of $i$, it follows that $z$ is $M$-visible, leading to the conclusion that  $y$ is also $M$-visible, contradicting the assumption.
 
 Therefore, we may assume that vertices of $L^x_{i-1}$ belong to $M$. 
 On other hand,  by Proposition   
 \ref{sosednji}, there exists a vertex $z \in L^x_{i-2}$ that does not belong to $M$.  
 Since $d(u,z)=2$, vertices $y$ and $z$ admit two common neighbours from the set  $L^x_{i-1} \subset M$, leading to a contradiction.
\end{proof}

\begin{thm} \label{dmv_thm}
\begin{displaymath}
\mu_d(Q_h) =
        \left \{ \begin{array}{llll}
              2,  &  h = 1   \\
              3,  &  h = 2   \\
              4,  &  h = 3   \\
              8,  &  h = 4   \\
              10,  &  h = 5   \\
              20,  &  h = 6   \\
              29,  &  h = 7   \\
             \end{array}. \right.
\end{displaymath}
Moreover, if $h \ge 8$, then $ 2 \cdot A(h,4) \le \mu(Q_h) \le  29 \cdot  2^{h-7} $.
\end{thm}

 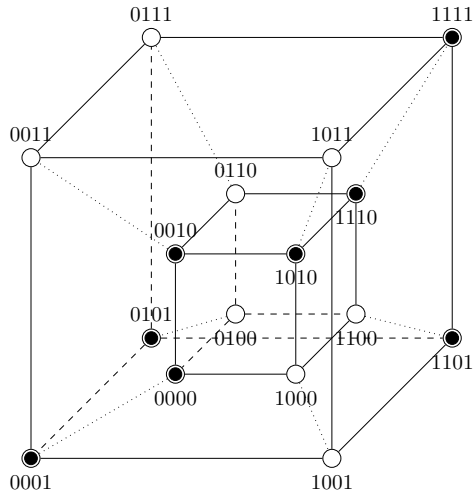
\begin{figure}[bt] 
 \centering
 \scalebox{0.8}
 { \tikzstyle{nod}= [circle, draw,inner sep=0pt, minimum size=0.3cm]
\begin{minipage}{\textwidth}
\begin{center}
\begin{tikzpicture}
\pgfsetxvec{\pgfpoint{1.cm}{0.0cm}}
\pgfsetyvec{\pgfpoint{0.0cm}{1.0cm}}
\node[nod] at (0,0) [label=below:$0001$] (0001) {};
\node[nod] at (5,0) [label=below:$1001$] (1001) {};
\node[nod] at (0,5) [label=above:$0011$] (0011) {};
\node[nod] at (5,5) [label=above:$1011$] (1011) {};
\node[nod] at (2,2) [label=above:$0101$] (0101) {};
\node[nod] at (7,2) [label=below:$1101$] (1101) {};
\node[nod] at (2,7) [label=above:$0111$] (0111) {};
\node[nod] at (7,7) [label=above:$1111$] (1111) {};
\node[nod] at (2.4,1.4) [label=below:$0000$] (0000) {};
\node[nod] at (4.4,1.4) [label=below:$1000$] (1000) {};
\node[nod] at (2.4,3.4) [label=above:$0010$] (0010) {};
\node[nod] at (4.4,3.4) [label=below:$1010$] (1010) {};
\node[nod] at (3.4,2.4) [label=below:$0100$] (0100) {};
\node[nod] at (5.4,2.4) [label=below:$1100$] (1100) {};
\node[nod] at (3.4,4.4) [label=above:$0110$] (0110) {};
\node[nod] at (5.4,4.4) [label=below:$1110$] (1110) {};
\path (0011)
edge (1011)
edge (0111)
edge (0001)
(1001)
edge (0001)
edge (1101)
edge (1011)
(1111)
edge (1101)
edge (1011)
edge (0111)
(0010)
edge (1010)
edge (0110)
edge (0000)
(1000)
edge (0000)
edge (1100)
edge (1010)
(1110)
edge (1100)
edge (1010)
edge (0110);
\path[dashed] (0101)
edge (1101)
edge (0001)
edge (0111)
(0100)
edge (1100)
edge (0000)
edge (0110);
\path[dotted] (0000) edge (0001)
(0010) edge (0011)
(0100) edge (0101)
(0110) edge (0111)
(1000) edge (1001)
(1010) edge (1011)
(1100) edge (1101)
(1110) edge (1111);

\filldraw (2.4,1.4) circle (3pt);
\filldraw (0,0) circle (3pt);
\filldraw (2.4,3.4) circle (3pt);
\filldraw (2,2) circle (3pt);
\filldraw (4.4,3.4) circle (3pt);
\filldraw (7,2) circle (3pt);
\filldraw (5.4,4.4) circle (3pt);
\filldraw (7,7) circle (3pt);

\end{tikzpicture}
\end{center}
\end{minipage}
}
\caption{A dual mutual-visibility set of $Q_4$} \label{dual_fig}
\end{figure}

\begin{proof}
For $h \le 6$,  we determined dual mutual-visibility numbers of $h$-cubes    using an ILP model and a reduction to SAT, as presented in Section 2. 

Additionally, we determined a dual mutual-visibility set of $Q_7$ with 29 vertices. 
However, the applied models did not rejected  the existence of a larger 
dual mutual-visibility set since the computations 
did not finish in a reasonable time.  

Thus, in order to confirm that a dual mutual-visibility set of cardinality 30 
does not exist, some additional restrictions where needed. 

First, we search (by using an ILP model) a largest  outer mutual-visibility set $M$ of $Q_7$ such that triples $u,v,z \in M$ with 
$Q_7[\{u,v,z\}]$ isomorphic to $K_{1,2}$ are forbidden, and we established that a largest  dual mutual-visibility set of $Q_7$ with this restriction has a size of 16. 

This result indicates that every dual mutual-visibility set $M$ of $Q_7$ of size greater than 16  must contain at least one triple $u,v,z \in M$ with 
$Q_7[\{u,v,z\}]$ isomorphic to $K_{1,2}$.
This fact allowed us to preselect  the vertices 0000000, 
0000001 and 0000010 to be included in $M$ in the applied SAT model. Moreover, we restricted the length of the paths involved in the SAT formulas to 4. 
The obtained model comprises 94110 variables and 244892 clauses.
We confirmed (the computation lasted 5562 seconds on a powerful desktop computer with 6 threads) that the  above presented reduction is unsatisfiable.
Hence, we may establish that an dual mutual-visibility set of $Q_7$ with 30 vertices does not exist. 

By definition, every total mutual-visibility set is also a dual 
 mutual-visibility set. Hence, the lower bound readily follows from Theorem
 \ref{total}. 
 Since the upper bound on $\mu_o(Q_h)$ for $h\ge 8$ follows from the fact 
that $\mu_d(Q_d)=29$ and Observation 
\ref{observation}, the proof is complete.
\end{proof}

It is worth to mention that the existence of a dual mutual-visibility set 
of cardinality $k$ in $Q_h$ does not guarantee that there exists  a dual mutual-visibility set of cardinality $k-1$ in $Q_h$. 
In particular, there is no dual mutual-visibility set of cardinality 7 in 
$Q_4$ although the dual mutual-visibility  number of $Q_4$ is 8. 
An example of a largest dual mutual-visibility set of $Q_4$ is presented in  Fig.  \ref{dual_fig}. 

\subsection{Summary of results and concluding remarks}
This section summarize  all known values and bounds on the 
mutual-visibility, dual mutual-visibility and outer mutual-visibility 
number of hypercubes. The corresponding values as well as total 
mutual-visibility  numbers for hypercubes of dimensions up to 11  are presented in Table \ref{t1}. 
(As already noted, the
problem of finding a largest total  mutual-visibility set in a $h$-cube is equivalent to finding the maximum size of a binary code of length $h$ and minimum Hamming distance 4. For that reason, this invariant is presented more in detail in Subsection 3.2.) 
With respect to the previous subsections, the table presents 
some additional lower bounds on the invariants studied in this paper. 
These bounds were improved by using a reduction to SAT.  
We provide the results of our computations, including the cardinalities of the obtained sets (as well as the time needed and the number of threads used), in the sequel.

\begin{table}[h]	
{\small
\begin{center}
\caption{Lower bounds, upper bounds and exact values on mutual-visibility number varieties } \label{t1}
\bgroup
\def\arraystretch{1.4}
\begin{tabular}{|c||c|c|c|c| }
  \hline
  $h$  & $\mu(Q_h)$ & $\mu_t(Q_h)$ & $\mu_d(Q_h)$ & $\mu_o(Q_h)$ \\ 
  \hline
  \hline
  3 &  5  & 2 & 4 & 4 \\
\hline
  4 &  9  &  4 &  8 &  6  \\
\hline
  5 &  16  & 4 &  10 & 12 \\
\hline
  6 &  32  & 8  &  20 & 24 \\
\hline
  7 &  59  & 16 &  29  &  40 \\
\hline
  8 & 116-118  & 32 &  52-58 &  72-80 \\
\hline
  9 &  222-236  & 40 &  86-116 & 126-160 \\
\hline
  10 &  432-472  &  80 &  148-232 & 252-320 \\
\hline
  11 &  820-944  & 144 &  210-464 & 462-640 \\
\hline
\end{tabular}
\egroup
\end{center}
}
\end{table}

In particular, we found 
a mutual-visibility set of $Q_8$ of size 116 (2 seconds, 2 threads),
a dual mutual-visibility set of $Q_8$  of size 52 (27140 seconds, 2 threads), 
an outer mutual-visibility set  of $Q_8$  of size 72 (3100 seconds, 10 threads), 
a mutual-visibility set of $Q_9$ of size 222 (840 seconds, 2 threads),
a dual mutual-visibility set of $Q_9$  of size 86 ($7.7 \cdot 10^5$ seconds, 10 threads), 
a mutual-visibility set of $Q_{10}$ of size 432  (4500 seconds, 32 threads),
a dual mutual-visibility set of $Q_{10}$ of size 148 (610000 seconds, 8 threads), a mutual-visibility set of $Q_{11}$ of size 820 (23400 seconds, 2 threads) and 
a dual mutual-visibility set of $Q_{11}$ of size 210 (27600 seconds, 16 threads).

In order to narrow down the search space for hypercubes of dimensions 8 or higher, we employed two additional heuristics, leading to the discovery of some large mutual-visibility and dual mutual-visibility sets.

In the first heuristic, we search a mutual-visibility set $M$ (or its variety) by presetting a "substantial" subset of vertices of $Q_{h}$ to 
be included in M in the applied SAT model. 
In this context, we may utilize the fact that every total mutual-visibility set also functions as a (standard) mutual-visibility set, as well as an outer and dual mutual-visibility set. Additionally, every outer (or dual) mutual-visibility set is also a (standard) mutual-visibility set. 
However, the best results in searching large mutual-visibility sets 
of  $Q_{h}$, $h\ge 8$, by this approach were obtained by preselecting 
vertices of the set $L_{\lfloor \frac{h}{ 2} \rfloor - 1}^v \cup 
L_{\lfloor \frac{h}{ 2} \rfloor + 2}^v$, where $v=0^h$.


The second heuristic leveraged the observation that for $h \le 6$, every vertex $u=u_1, \ldots, u_h$ in the largest computed dual mutual-visibility set $M$ of $Q_h$ has its corresponding \textit{antipode} in $M$, denoted as 
$v=v_1, \ldots, v_h$, where $v_i = 1 - u_i$, $i \in [h]$. 
Inspired by this, we devised a SAT model wherein the antipode of every vertex in the form $u=u_1, \ldots, u_{h-1}0$ included in $M$ is always inserted into $M$. This approach significantly accelerated the search for large dual mutual-visibility sets. Notably, we managed to find dual mutual-visibility sets of $Q_{10}$ and $Q_{11}$  with 148 and 210 vertices, respectively.

\subsection*{Data availability}
The datasets generated during and/or analysed during the current study are available on the web page \url{https://omr.fnm.um.si/wp-content/uploads/2024/04/hypercubesMV.pdf} and from the corresponding author on reasonable request.

\subsection*{Funding}
This work was supported by the Slovenian Research Agency under the grants P1-0297, J1-2452, J1-1693 and  J2-1731.



\end{document}